\documentclass[a4paper]{article}

\usepackage[english]{babel}
\usepackage[utf8x]{inputenc}
\usepackage[T1]{fontenc}
\usepackage{latexsym,amssymb,amsmath,amsfonts,amscd} 

\usepackage[a4paper,top=3cm,bottom=2cm,left=3cm,right=3cm,marginparwidth=1.75cm]{geometry}

\usepackage{amsmath}
\usepackage{amsthm}
\usepackage{graphicx}
\usepackage[colorinlistoftodos]{todonotes}
\usepackage[colorlinks=true, allcolors=blue]{hyperref}

\newtheorem{theorem}{Theorem}
\newtheorem{lema}[theorem]{Lemma}
\newtheorem{corollary}[theorem]{Corollary}
\newtheorem{proposition}[theorem]{Proposition}
\theoremstyle{definition}
\newtheorem{definition}{Definition}
\theoremstyle{remark}
\newtheorem{example}{Example}
\newtheorem{obs}{Remark}

\title{$\mathbb{Z}$-graded polynomial identities of the Grassmann algebra}
\author{Alan de Ara\'ujo Guimar\~aes\thanks{Supported by PhD Grant from CNPq, Brazil, and by CNPq grant No. 421129/2018-2}, Plamen Koshlukov\thanks{Partially supported by FAPESP grant No. 2018/23690-6 and by CNPq grant No. 302238/2019-0.}\\
Department of Mathematics, State University of Campinas\\
651 Sergio Buarque de Holanda, 13083-859 Campinas, SP, Brazil\footnote{A. A. Guimar\~aes' current address: Department of Mathematics, Federal University of Rio Grande do Norte, 59078-970 Natal, RN, Brazil}\\
e-mails: \texttt{alansimoes10@hotmail.com}, \texttt{plamen@unicamp.br}
}

\begin{document}
\maketitle

\begin{abstract}
Let $F$ be an infinite field of characteristic different from 2, and let $E$ be the Grassmann algebra of an infinite dimensional $F$-vector space $L$. In this paper we study the $\mathbb{Z}$-graded polynomial identities of $E$ with respect to certain $\mathbb{Z}$-grading such that the vector space $L$ is homogeneous in the grading. More precisely, we construct three types of $\mathbb{Z}$-gradings on $E$, denoted by $E^{\infty}$, $E^{k^\ast}$ and $E^{k}$, and we give the explicit form of the corresponding $\mathbb{Z}$-graded polynomial identities. We show that the homogeneous superalgebras $E_{\infty}$, $E_{k^\ast}$ and $E_{k}$ studied in  \cite{disil} can be obtained from $E^{\infty}$, $E^{k^\ast}$ and $E^{k}$ as quotient gradings. Moreover we exhibit several other types of homogeneous $\mathbb{Z}$-gradings on $E$, and describe their graded identities. 
\end{abstract}



\section{Introduction}
Let $F$ be an infinite field of characteristic different from 2. If $L$ is a vector space over $F$ with basis $e_{1}$, $e_{2}$, \dots, the infinite dimensional Grassmann algebra $E$ of $L$ over $F$ is the vector space with a basis consisting of 1 and all products $e_{i_1}e_{i_2}\cdots e_{i_k}$ where $i_{1}<i_{2}<\cdots <i_{k}$, $k\geq 1$. The multiplication in $E$ is induced by $e_{i}e_{j}=-e_{j}e_{i}$ for all $i$ and $j$. We shall denote by $B=B_{E}$ the above canonical basis of $E$. The Grassmann algebra $E$ is one of the most important algebras satisfying a polynomial identity, also known as PI-algebras. Recall that in characteristic 2 the Grassmann algebra is commutative hence not very interesting from the point of view of PI theory. The polynomial identities of the Grassmann algebra were described by Latyshev \cite{latyshev}, and later on by Krakowski and Regev \cite{KR}. Recall that in the latter paper the structure of the ideal of identities of $E$ was described very tightly. The identities of $E$ have been extensively studied also in positive characteristic. The interested reader can consult the paper \cite{agpk} and the references therein for further information. In characteristic 0, it is the easiest example of an algebra which does not satisfy any standard identity $s_n$. We recall that the standard polynomial $s_n(x_1,\ldots, x_n) = \sum_{\sigma\in S_n} (-1)^\sigma x_{\sigma(1)}\cdots x_{\sigma(n)}$ is the alternating sum of all monomials obtained by $x_1\cdots x_n$ by permuting its variables. Here $S_n$ stands for the symmetric group on the letters $\{1,\ldots, n\}$, and $(-1)^\sigma$ is the sign of the permutation $\sigma$. This polynomial plays a prominent role in PI theory, especially in characteristic 0. We recall that a theorem of A. Kemer \cite{kemer_vol} asserts that every PI algebra over a field of characteristic $p>2$ satisfies some standard identity.

The Grassmann algebra has a natural $\mathbb{Z}_2$-grading $E=E_0\oplus E_1$ where $E_0$ is the span of 1 and all products $e_{i_1}\cdots e_{i_k}$ with even $k$ while $E_1$ is the span of the products with odd $k$. Clearly $E_0$ is just the centre of $E$ and $E_1$ is the ``anticommuting'' part of $E$. This natural structure of a $\mathbb{Z}_2$-graded algebra on $E$ makes the Grassmann algebra very important not only in Algebra but in various areas of Mathematics as well as in Theoretical Physics. We are not going to discuss the applications of the Grassmann algebra here; we point out to the book by Berezin \cite{berezin} for such applications. 

Kemer \cite{basekemer, basekemer2} developed the structure theory of the ideals of identities of associative algebras, also called T-ideals, in characteristic 0. This description led Kemer to the positive solution of the long standing and extremely important problem posed by W. Specht in 1950: Is every T-ideal in characteristic 0 finitely generated as a T-ideal? A key ingredient in Kemer's research was the following result: Every associative PI-algebra over a field of characteristic zero is PI-equivalent to the Grassmann envelope of a finite dimensional associative superalgebra. Here we recall that if $A=A_0\oplus A_1$ is a superalgebra then its Grassmann envelope is defined as $E(A) = A_0\otimes E_0\oplus A_1\otimes E_1$. 

With some abuse of notation we shall use the term superalgebra as synonymous to $\mathbb{Z}_2$-graded algebra. While in the associative case these two notions do coincide they do not if one studies nonassociative algebras. Thus a Lie (or Jordan) superalgebra is a $\mathbb{Z}_2$-graded algebra $A$ such that $E(A)$ is a Lie (or Jordan) algebra. Observe that $A$ need not be a Lie (or Jordan) algebra at all, and in all interesting cases it is not.

Group gradings on algebras and the corresponding graded polynomial identities have been extensively studied in PI theory during the last two or three decades. As we already mentioned the Grassmann algebra admits a natural grading by the group $\mathbb{Z}_{2}$; here we shall denote it by $E_{can}=E_{(0)}\oplus E_{(1)}$. The graded identities for the canonical grading on $E$ are well known, see for example \cite{GMZ}. In the last fifteen years a substantial number of papers has  presented results on gradings on the Grassmann algebra and their graded identities. However two strong hypotheses have always been done:

\begin{itemize}
\item The group $G$ of the grading on $E$ is finite. 
\item All generators $e_{1}$, $e_{2}$, \dots, $e_{n}$, \dots{} of $E$ are  homogeneous in the grading.
\end{itemize} 
We cite here the research developed by Anisimov \cite{anisimov1, anisimov2}, and by Di Vincenzo and da Silva \cite{silcod, disil}. 

In a previous paper we studied $\mathbb{Z}_2$-gradings on $E$ without assuming the homogeneity of the generators, see \cite{aagpk}. We proved that there exist very many nonhomogeneous such gradings on $E$. On the other hand it turns out that in all "typical" cases the gradings obtained are very similar to homogeneous ones.  

In this paper we shall dispense with the first of the above two conditions, and we shall study gradings on $E$ by the infinite cyclic group $\mathbb{Z}$. We point out that until now the only known $\mathbb{Z}$-grading on $E$ has been the natural one. It is given by $E^{can}= \oplus_{\substack{n\in\mathbb{Z}}} E^{(n)}$ where
\[
E^{(n)}=\begin{cases} 0,\text{ if }  n<0 \\ F, \text{ if } n=0 \\ span_{F}\{w\in B_{E}\mid  |supp(w)|=n \}, \text{ if } n\geq 1
\end{cases}.
\]
Gradings by $\mathbb{Z}$ have been studied in the context of polynomial identities. We recall that Vasilovsky \cite{vsl98, vsl99} described the $\mathbb{Z}$-grading and $\mathbb{Z}_{n}$-grading on matrix algebras in characteristic 0, and later on Azevedo \cite{ssa2, ssa1} obtained similar results over any infinite field. The $\mathbb{Z}$-graded identities (as well as the $\mathbb{Z}_2$-graded ones) for the Lie algebra $sl_2(F)$ were described in \cite{pksl2}. In all these cases the graded polynomial identities are essentially the same as in the "finite" case. In \cite{jfakpk} the natural $\mathbb{Z}$-grading on the Lie algebra $W_1$ of the derivations on the polynomial ring in one variable was studied. In the latter paper it was proved that the graded identities for $W_1$ do not admit any finite generating set. (We recall here that finding a generating set for the ordinary identities of $W_1$ is an important open problem in PI theory.) 

Let us mention that the study of graded algebras started with the natural grading by $\mathbb{Z}$ on the polynomial ring. Gradings by $\mathbb{Z}$ where only finitely many components are nonzero are important in various areas of Algebra: in the theory of finite dimensional Lie algebras, in the theory of Jordan algebras, pairs and triple systems, and so on. 

We return to the Grassmann algebras. 
When we pass to the quotient grading modulo $2\mathbb{Z}$ we  see that $E_{can}$ is induced by $E^{can}$. Hence we have the correspondence 
\[
 E^{can} \stackrel{\pmod{2\mathbb{Z}}}{\longmapsto} E_{can}.
\]
Di Vincenzo and da Silva in \cite{disil} described the $\mathbb{Z}_{2}$-gradings on the Grassmann algebra $E$ in characteristic zero. They assumed that the vector space $L$ is homogeneous. There exist three possible structures, namely:
\[
\|e_{i}\|_{k}=\begin{cases} 0,\text{ if }  i=1,\ldots,k\\ 1, \text{  otherwise} 
\end{cases},
\]	
\[
\|e_{i}\|_{k^\ast}=\begin{cases} 1,\text{ if }  i=1,\ldots,k\\ 0, \text{ otherwise } 
\end{cases},
\]	
and
\[
\|e_{i}\|_{\infty}=\begin{cases} 0,\text{ if }  i\text{ even }\\ 1, \text{ otherwise } 
\end{cases}.
\]
These three cases provide us with three superalgebras denoted respectively by $E_{k}$, $E_{\infty}$, and $E_{k^\ast}$. The $\mathbb{Z}_{2}$-graded polynomial identities of $E_{k}$, $E_{\infty}$, and $E_{k^\ast}$ were completely described in \cite{disil}. In a similar way in \cite{centroneP} the author described the polynomial identities of $E_{k}$, $E_{\infty}$ and $E_{k^\ast}$ over an infinite field of positive characteristic $p>2$, and in \cite{LFG}, the same was done when $F$ is a finite field.

Following the above discussion here we develop a rather general procedure for producing $\mathbb{Z}$-gradings on $E$ such that the vector space $L$ is homogeneous. In fact such a procedure can be used to obtain the gradings by an arbitrary abelian group $G$. Moreover the construction allows us to realize the superalgebras $E_{k}$, $E_{\infty}$ and $E_{k^\ast}$ as quotient gradings from appropriate and natural $\mathbb{Z}$-gradings on $E$.

In addition to the general construction we deal with three more particular types of $\mathbb{Z}$-gradings on $E$. Our interest in these three types of $\mathbb{Z}$-gradings derives from the fact that they yield, in a sense, the $\mathbb{Z}_2$-gradings considered by Di Vincenzo and da Silva in \cite{disil} when passing to the quotient of $\mathbb{Z}$ by $2\mathbb{Z}$. We describe the respective $\mathbb{Z}$-graded polynomial identities. First we study the situation where the base field is of characteristic zero; afterwards we extend our results to an infinite field of characteristic different from 2. 

\section{Preliminaries}
Let $F$ be a fixed infinite field of characteristic different from 2. Throughout this paper all vector spaces and algebras will be associative and with unity, and will be considered over $F$. We also fix an abelian group $G$. Whenever necessary we shall state explicitly any restrictions on $F$ and/or $G$ which we impose.

A $G$-grading on an associative algebra $A$ is a vector space decomposition $A=\oplus_{g\in G}A_{g}$ such that  $A_{g}A_{h}\subset A_{gh}$ for all $g$, $h\in G$. When $a\in A_{g}$ we say that $a$ is homogeneous and its degree is $\alpha(a)=g$. In some occasions we will want to use a subscript (or superscript) to the degree, in such cases we will use $\|a\|$ for the homogeneous degree of $a$, like $\|a\|_i$ or $\|a\|^i$. If $H$ is a subgroup of $G$ we define the quotient $G/H$-grading as $A=\oplus_{\overline{g}\in G/H}A_{\overline{g}}$ where $A_{\overline{g}}=\oplus_{h\in H}A_{gh}$.

Group gradings on the Grassmann algebra with groups $G$ other than $\mathbb{Z}_2$ were studied in the papers \cite{disilplamen, centroneG}. In these the authors describe homogeneous gradings on $E$ by cyclic groups of prime order and by finite abelian groups, respectively. 

Let $ X=\cup_{i\in\mathbb{Z}}X_{i}$ be the disjoint union of infinite countable sets of variables $X_{i}=\{x_{1}^{i},x_{2}^{i},\ldots\}$, $i\in\mathbb{Z}$. Assuming that for each $i\in\mathbb{Z}$ the elements of the sets $X_{i}$ are of $\mathbb{Z}$-degree $i$, the free associative algebra $F\langle X|\mathbb{Z}\rangle$ has a natural $\mathbb{Z}$-grading $\oplus_{i\in\mathbb{Z}} F^{i}$. Here $F_{i}$ is the vector subspace of $F\langle X|\mathbb{Z}\rangle$ spanned by all monomials of $\mathbb{Z}$-degree $i$. It is immediate to see that this defines indeed a $\mathbb{Z}$-grading on $F\langle X|\mathbb{Z}\rangle$. It is also easy to see that $F\langle X|\mathbb{Z}\rangle$ is free in the following sense. Given a $\mathbb{Z}$-graded algebra $A=\oplus_{i\in\mathbb{Z}} A_i$ and a map $f\colon X\to A$ such that $f(X_i)\subseteq A_i$ for every $i$ then $f$ can be uniquely extended to an algebra homomorphism $\varphi\colon F\langle X|\mathbb{Z}\rangle\to A$ which respects the $\mathbb{Z}$-gradings on these algebras. (Such homomorphisms are called $\mathbb{Z}$-graded ones.)

A polynomial $f(x_{j_1}^{l_1},\ldots, x_{j_r}^{l_r})$ is called a $\mathbb{Z}$-graded polynomial identity (PI) of the $\mathbb{Z}$-graded algebra $A$ if $f(a_{1},\ldots, a_{r})=0$ for all $a_{l}\in A_{i_l}$. The set $T_{\mathbb{Z}}(A)$ of all $\mathbb{Z}$-graded polynomial identities is an ideal of $F\langle X|\mathbb{Z}\rangle$. It is also invariant under $\mathbb{Z}$-graded endomorphisms of $F\langle X|\mathbb{Z}\rangle$. Such ideals are called $T_{\mathbb{Z}}$-ideals. As in the case of ordinary polynomial identities one proves that an ideal $I$ in $F\langle X|\mathbb{Z}\rangle$ is a $T_{\mathbb{Z}}$-ideal if and only if it coincides with the ideal of all $\mathbb{Z}$-graded identities of some $\mathbb{Z}$-graded algebra $A$. It is well known that studying ordinary polynomial identities in characteristic 0, one may consider the multilinear ones. The analogous fact holds for graded identities as well. Thus if $A$ is a $\mathbb{Z}$-graded algebra over the field $F$ of characteristic 0, the ideal $T_{\mathbb{Z}}(A)$ of all $\mathbb{Z}$-graded identities of $A$ is generated as a $T_{\mathbb{Z}}$-ideal by its multilinear polynomials. In the more general case of an infinite field one has to take into account the multihomogeneous polynomials instead of the multilinear ones. 

In the case of ordinary identities of unitary algebras one can reduce further the set of polynomial identities that determine a given T-ideal. Such a reduction is given by the proper polynomials. 

Let $L(X)$ be the free Lie algebra freely generated over $F$ by the set $X$. It is well known that its universal enveloping algebra is the free associative algebra $F\langle X\rangle$ (Witt's Theorem). Take an ordered basis of $L(X)$ such that the variables from $X$ precede the longer commutators. By the Poincar\'e--Birkhoff--Witt theorem $F\langle X\rangle$ has a basis as a vector space consisting of elements 
\[
x_{i_1} \cdots x_{i_m} u_{j_1} \cdots u_{j_n}
\]
where $i_1\le \cdots\le i_m$, all $u_{j_k}$ are commutators of degree $\ge 2$, and moreover $j_1\le \cdots\le j_n$, $m$, $n\ge 0$. Let $A$ be a unitary associative algebra and assume $f\in T(A)$ is an identity for $A$. A standard argument, see for example \cite[Section 4.3]{bookdrensky} shows that $f$ is equivalent as an identity to a finite collection of polynomials which are linear combinations of products of commutators, as long as the field $F$ is infinite. Such polynomials (linear combinations of products of commutators) are called \textsl{proper polynomials}. In particular, if $F$ is of characteristic 0, and $1\in A$ then $T(A)$ is generated as a T-ideal by its proper multilinear polynomials. The idea of proof of this fact consists in substituting the variables $x$ in $f$ by $x+1$, and then taking the homogeneous components. As the field is infinite, these homogeneous components are identities for $A$ as well. But if one substitutes 1 inside a commutator then the corresponding summand vanishes. Thus one eliminates the variables outside of commutators. 

The above argument does not apply straightforward to graded identities. But a slight modification of it shows that one can still make use of it, namely by considering the so-called 0-proper polynomials. These are polynomials where each variable of homogeneous degree 0 appears in commutators only, while the remaining variables may appear inside or outside the commutators. Here we draw the readers' attention that the element 1 of the algebra lies in the homogeneous component of degree 0. The 0-proper polynomials have been used by various authors, see for example \cite[Proposition 1.2]{opa}. 

We make use of the 0-proper polynomials in the description of the graded identities of the canonical $\mathbb{Z}$-grading on $E$, and in an essential way, in Section~\ref{sectEk} where we describe the identities of the grading $E^k$ which is the analogue to $E_k$.

\section{The natural $\mathbb{Z}$-grading on $E$}

Recall that the natural $\mathbb{Z}$-grading on $E$ is defined as  $E^{can}= \oplus_{\substack{n\in\mathbb{Z}}} E^{(n)}$ where
\[
E^{(n)}=\begin{cases} 0,\text{ if }  n<0 \\ F, \text{ if } n=0 \\ span_{F}\{w\in B_{E}\mid |supp(w)|=n \}, \text{ if } n\geq 1
\end{cases}.
\]
We recall that if $F\langle X|\mathbb{Z}\rangle$ is the free associative $\mathbb{Z}$-graded algebra we denote the degree of the variable $x$  by $\alpha (x)$.

We describe now the generators of the $\mathbb{Z}$-graded polynomial identities for $E^{can}$.

\begin{lema} 
	The following polynomials are $\mathbb{Z}$-graded identities of $E^{can}= \oplus_{\substack{n\in\mathbb{Z}}} E^{(n)}$: 
	\begin{itemize}
		\item $x$, if $\alpha(x)<0$;
		\item $[x_{1},x_{2}]$, if $\alpha(x_{1})$ or $\alpha(x_{2})$  is an even integer;
		\item $x_{1}x_{2} + x_{2}x_{1}$, if $\alpha(x_{1})$ and $\alpha(x_{2})$ are odd integers. 
		\label{eq:lema 1}
			\end{itemize}
\end{lema}
\begin{proof}
The proof is trivial and hence omitted.
\end{proof}

\begin{proposition}
Let the field $F$ be of characteristic 0. 
	The $T_{\mathbb{Z}}$-ideal of the $\mathbb{Z}$-graded polynomial identities  of $E^{can}=\oplus_{n\in\mathbb{Z}} E^{(n)}$ is generated by the graded polynomials  from Lemma~\ref{eq:lema 1}. 
\end{proposition}

\begin{proof}
	Let $I$ be the $T_{\mathbb{Z}}$-ideal generated by the graded identities from 
Lemma~\ref{eq:lema 1}. The inclusion $I\subset T_{\mathbb{Z}}(E^{can})$ is immediate by Lemma~\ref{eq:lema 1}. Thus we have to prove only the opposite inclusion. 
	
Let us take a multilinear polynomial $f(x_{1},\ldots, x_{t}, x_{t+1},\ldots, x_{t+s})\notin I$ where $\alpha(x_{i})$ is an even integer whenever $i=1$, \dots, $t$, and $\alpha(x_{i})$ is odd for $i=t+1$, \dots, $t+s$. Since $f$ is not in $I$ we assume that each  variable is of degree $\ge 0$. Modulo the graded identities from Lemma \ref{eq:lema 1} we obtain that
	\[
	f\equiv\beta x_{1}\ldots x_{t} x_{t+1}\ldots x_{t+s} \pmod {I},\text{ with } \beta\neq 0.
	\]
Now we consider the following $\mathbb{Z}$-graded evaluation
	\begin{align*}
x_{1} & \mapsto  e_{1}\ldots e_{\alpha(x_{1})}, \\
x_{2} & \mapsto  e_{\alpha(x_{1})+1}\ldots e_{\alpha(x_{1})+\alpha(x_{2})}, \\
	&\vdots \\
x_{t} & \mapsto  e_{\alpha(x_{1})+ \cdots+ \alpha(x_{t-1})+1}\ldots e_{\alpha(x_{1})+\alpha(x_{2})+ \cdots + \alpha(x_{t-1})+\alpha(x_{t})}.
\end{align*}
Putting $p=\alpha(x_{1})+\alpha(x_{2})+\cdots +\alpha(x_{t})$, we define further
\begin{align*}
x_{t+1} & \mapsto e_{p+1}\ldots e_{p+\alpha(x_{t+1})},\\ 
x_{t+2} & \mapsto e_{p+\alpha(x_{t+1})+1}\ldots e_{p+\alpha(x_{t+1})+\alpha(x_{t+2})}, \\
	&\vdots \\
x_{t+s} & \mapsto e_{p+\alpha(x_{t+1})+ \cdots + \alpha(x_{t+s-1})+1}\ldots e_{p+\alpha(x_{t+1})+ \cdots + \alpha(x_{t+s-1}) +\alpha(x_{t+s})}.
\end{align*}
This evaluation is clearly $\mathbb{Z}$-graded. The products of the generators $e_i$ involved in it have disjoint supports, and therefore their product will not vanish on $E$.  
\end{proof}

In the following theorem we deal with the case when $F$ is an infinite field of characteristic $p>2$.

\begin{theorem}
Over an infinite field $F$ of characteristic $p>2$, all $\mathbb{Z}$-graded polynomial identities of $E^{can}$ are consequences of the graded identities from Lemma~\ref{eq:lema 1} together with the identity $x^p$ whenever $\alpha(x)\ge 1$ is an even integer.
\end{theorem}

\begin{proof}
Let $f(x_{1},\ldots,x_{n})$ be a 0-proper multihomogeneous identity of $E^{can}$. 
As $x$ is a graded identity whenever $\alpha(x)<0$ we can assume that  all variables $x$ appearing in $f$ are of degree $\alpha(x)>0$. The identity $x_{1}x_{2} + x_{2}x_{1}=0$ when $x_{1}$ and $x_{2}$ are of odd degrees  gives us $x^2=0$ if $\alpha(x)$ is an odd integer.  Combining this with the fact that $[x,y]$ is a graded identity when $\alpha(x)$ is an even integer,  we can suppose that all odd variables in $f$ are multilinear. On the other hand the identity $x^{p}=0$ when $\alpha(x)$ is even allows us to assume that 
\[
f(x_{1},\ldots,x_{n})=\beta x_{1}^{l_1}\cdots x_{v}^{l_v}x_{v+1}\cdots x_{n}
\]
where $\alpha(x_{i})$ is even if $i=1$, \dots, $v$, $\alpha(x_{i})$ is odd if $i=v+1$, \dots, $n$, and $l_{i}<p$ for every $i=1$, \dots, $v$.
Therefore $f$ vanishes on $E^{can}$ if and only if $\beta=0$.
\end{proof}

It is well known that the non-unitary Grassmann algebra in characteristic $p>2$ satisfies the identity $x^p$, see for example \cite[Lemma 12]{agpk} and the references therein. This implies that $a^p=0$ for every homogeneous element $a\in E^{can}$ of degree $\ge 1$. On the other hand if $\alpha(a)$ is odd then by Lemma~\ref{eq:lema 1} we have $2a^2=0$ and $a^2=0$ thus $a^p=0$ as well. Hence one has to impose the identity $x^p$ only for homogeneous elements of even degree. 
	
\section{$\mathbb{Z}$-gradings on $E$: A general construction}

In this section we provide the general method of constructing $\mathbb{Z}$-gradings on $E$. In fact the method works for every abelian group $G$ instead of $\mathbb{Z}$. 

We start with defining three types of $\mathbb{Z}$-gradings on the Grassmann algebra $E$; these are the natural analogues of the ones considered in \cite{disil}. To this end we consider the following attribution of degrees on its generators:
\begin{align*}
\|e_{i}\|^{k} & =\begin{cases} 0,\text{ if }  i=1,\ldots,k\\ 1, \text{ otherwise } 
\end{cases},\\	
\|e_{i}\|^{k^\ast} & =\begin{cases} 1,\text{ if }  i=1,\ldots,k\\ 0, \text{ otherwise } 
\end{cases},\\
\|e_{i}\|^{\infty} & =\begin{cases} 0,\text{ for }  i\text{ even }\\ 1, \text{ for i odd } 
\end{cases}.
\end{align*}	
Then we induce the $\mathbb{Z}$-grading on $E$ putting 
\[
\|e_{j_{1}}\cdots e_{j_{n}}\| =\|e_{j_{1}}\|+\cdots + \|e_{j_{n}}\|,
\]
and extend it to $E$ by linearity.

We denote by $E^{k}$, $E^{k^\ast}$ and $E^{\infty}$ the three types of $\mathbb{Z}$-graded algebras above. Observe that the natural $\mathbb{Z}$-graded algebra $E^{can}$ can be obtained from the previous attribution of degrees, it is the case when $\|e_{i}\|=1$ for all $i$, that is it is the grading $E^0$. 

\begin{obs} 
The superalgebras $E_{k}$, $E_{k^\ast}$, and $E_{\infty}$ can be obtained respectively from $E^{k}$, $E^{k^\ast}$ and $E^{\infty}$ in a natural manner. Indeed, passing to the quotient grading by the subgroup $2\mathbb{Z}$ we note that
\begin{align*}
 E^{k} & \stackrel{\pmod{2\mathbb{Z}}}{\longmapsto} E_{k},\\
 E^{k^\ast} & \stackrel{\pmod{2\mathbb{Z}}}{\longmapsto} E_{k^\ast}, \\
 E^{\infty} & \stackrel{\pmod{2\mathbb{Z}}}{\longmapsto} E_{\infty}.
 \end{align*}
Clearly one can find infinitely many $\mathbb{Z}$-gradings on $E$ that produce the same $\mathbb{Z}_2$-grading. One can choose some ``natural'' $\mathbb{Z}$-grading producing a given $\mathbb{Z}_2$-grading; the same holds word by word if we consider gradings on $E$ by a cyclic group $\mathbb{Z}_m$ for every $m$. 
\end{obs}

\medskip

Now we present a general method for constructing $\mathbb{Z}$-gradings on the algebra $E$. 

Let $l\in\mathbb{N}$ be a positive integer and consider a list $(n_{1},\ldots, n_{l})$ where $n_{1}<\cdots < n_{l}$ and each $n_{j}$ belongs to $\mathbb{Z}$. We write the vector space $L$ as 
\[
L=L_{n_1}^{v_1}\oplus L_{n_2}^{v_2}\oplus\cdots\oplus L_{n_l}^{v_l}
\]
where we assume each $L_{n_j}^{v_j}\neq 0$. Suppose $\dim L_{n_j}^{v_j}=v_{j}$ where $v_{j}\in\mathbb{N}$, or $v_{j}=\infty$. Up to a change of the basis of the vector space $L$, we can assume that the generators $e_i$ of $E$ satisfy $e_i\in \cup_{j=1}^l L_{n_j}^{v_j}$. In other words we split the basis $e_1$, $e_2$, \dots{} of $E$ into $l$ disjoint sets (some of these may be finite). Denote by $L_{n_j}^{v_j}$ the span of the $j$-th set, and attribute homogeneous degree $n_j$ to the elements of the vector space $L_{n_j}^{v_j}$. 

Thus we induce a $\mathbb{Z}$-grading on $E$ by using the lists $(n_{1},\ldots, n_{l})$ and $(v_{1},\ldots, v_{l})$. We define
\[
\|e_{k}\|=n_{j}\quad \mbox{\textrm{ if and only if }}\quad  e_{k}\in L_{n_j}^{v_j}.
\]
For every product $e_{k_{1}}e_{k_{2}}\cdots e_{k_{s}}$ belonging to $E$ we define  
\[
\|e_{k_1}e_{k_2}\cdots e_{k_s}\|=\|e_{k_1}\|+\|e_{k_2}\|+\cdots +\|e_{k_s}\|.
\]
Hence the two lists $(n_{1},\ldots, n_{l})$ and $(v_{1},\ldots, v_{l})$ produce a $\mathbb{Z}$-grading on $E$. We shall denote this grading by $E_{(n_{1},\ldots, n_{l})}^{(v_{1},\ldots, v_{l})}$. 

\begin{obs}
Observe that $E^{can}$, $E^{k^\ast}, E^{\infty}$ and $E^{k}$ are obtained from the lists
\begin{itemize}
\item $(1)$ and $(\infty)$,
\item $(0,1)$ and $(\infty, k)$,
\item $(0,1)$ and $(\infty,\infty)$,
\item $(0,1)$ and $(k, \infty)$,
\end{itemize}
\noindent respectively. Hence the procedure described above is indeed a  general one.
\end{obs}

\begin{obs}
	The procedure just described is more general; it can be used to construct $G$-gradings on $E$, for every abelian group $(G, +)$. In fact, we can consider any map $\eta$ from $L$ to $G$ and define the $G$-degree of the element $e\in L$ by
	\[\|e\|_{G}=\eta(e).\]
	One defines
	\[
	\|e_{i_1}e_{i_2}\cdots e_{i_s}\|_{G}=\|e_{i_1}\|_{G}+\|e_{i_2}\|_{G}+\cdots +\|e_{i_s}\|_{G}
	\]
	and this produces a $G$-grading on $E$. 
\end{obs}

\begin{definition}
We say that a $\mathbb{Z}$-grading $E=\oplus_{n\in\mathbb{Z}} E_{n}$ on the Grassmann algebra is of finite coverage if there exists a list $(n_{1},n_{2},\ldots, n_{k})$ of integers such that $L\subset E_{n_{1}}\oplus E_{n_{2}}\oplus\cdots\oplus E_{n_{k}}$.
\end{definition}
\begin{obs}
Recall that if $A=\oplus_{g\in G} A_g$ is a $G$-graded algebra the support of the grading consists of the elements $g\in G$ such that $A_g\ne 0$. 
Clearly we have that
\begin{itemize}
\item Every $\mathbb{Z}$-grading on $E$ defined by a fixed list is a $\mathbb{Z}$-grading of finite coverage.
\item Every $\mathbb{Z}$-grading on $E$ of finite support is a $\mathbb{Z}$-grading of finite coverage. However the converse is not true. For example, the $\mathbb{Z}$-grading $E^{can}$ is of finite coverage (since $L\subset E^{(1)}$) but it does not have finite support.
\end{itemize}
\end{obs}

In what follows we will be interested in gradings on $E$ of finite coverage.

Assume $F$ is of characteristic 0. Let $V_n$ be the vector subspace of $F\langle X|\mathbb{Z}\rangle$ spanned by all multilinear monomials in $x_1$, $x_2$, \dots, $x_n\in X$ with all possible $\mathbb{Z}$-gradings on the $x_i$. Then it is clear that if $A$ is a $\mathbb{Z}$-graded algebra the $T_{\mathbb{Z}}$-ideal $T_{\mathbb{Z}}(A)$ is generated by all polynomials lying in the intersections $T_{\mathbb{Z}}(A)\cap V_n$, $n\ge 1$. 
The following lemma is standard. Its proof is easy; when the grading group is $\mathbb{Z}_2$ it is the contents of Lemmas 3 and 4 of \cite{disil}. 
\begin{lema} Let the field $F$ be of characteristic 0 and let $f(x_{1},\ldots, x_{n})\in V_{n}$ be multilinear. Then $f$ is an ordinary identity for $E$, that is $f\in T(E)$, if and only if for every map  $h\colon\{1,\ldots, n\}\to\mathbb{Z}_{2}$, there exist $a_{1}$, \dots, $a_{n}\in B_{E}$ with pairwise disjoint supports such that, for all $i$ satisfying $|supp(a_{i})|\equiv h(i)\pmod{2\mathbb{Z}}$ one has $f(a_{1},\ldots, a_{n})=0$. 
	\label{market1}
	\end{lema}

\section{Graded identities of $E^{k^\ast}$ and $E^{\infty}$}\label{Section5}

Let $k$ be a positive integer. We shall deal with the $\mathbb{Z}$-grading $E^{k^\ast}=\oplus_{n\in\mathbb{Z}}A_{n}$. In this case we have $e_{1}$, \dots, $e_{k}\in A_{1}$ and $e_{n}\in A_{0}$, for  $n>k$. Obviously, if $n>k$ or $n<0$ we have $A_{n}=0$. Therefore we have the following decomposition of $E^{k^*}$:
\[
E^{k^\ast}=A_{0}\oplus A_{1}\oplus\cdots\oplus A_{k}.
\]
We observe that for every $t=1$, \dots, $k$, a (nonzero) product of the generators $e_i$, say $w\in B_E$, belongs to the component $A_{t}$ if only if $w$ has exactly $t$ factors in the set $\{e_{1},\ldots, e_{k}\}$. Also if $1\le t\le k$ the homogeneous component of degree $t([k/t]+1)$ is trivial.

It is immediate to see that the algebra $E^{k^\ast}$ satisfies the following $\mathbb{Z}$-graded polynomial identities:
\begin{itemize}
	\item $x$, if $\alpha(x)\notin\{0,\ldots, k\}$.
	\item $[x_{1}, x_{2}, x_{3}]$, for all degrees $\alpha(x_{1})$, $\alpha(x_{2}$) and $\alpha(x_{3})$. 
\end{itemize} 

Given a positive integer $t$, $1\le t\le k$, we consider the non-negative integers $m$, $l_{1}$, \dots, $l_{t}$. Let us denote by $V_{m ,l_{1},\ldots, l_{t}}$ the vector space of $\mathbb{Z}$-graded multilinear polynomials having: 
	
\begin{tabular}{rll}
$m$ &  variables of degree 0, & namely  $z_{1}$, \dots, $z_{m}$,\\ 
$l_{1}$ & variables of degree  1, & namely $x_{1}^{1}$, \dots, $x_{l_{1}}^{1}$,\\
$\vdots$ &&\\ 
$l_{t}$ &  variables of degree $t$, & namely $x_{1}^{t}$, \dots, $x_{l_{t}}^{t}$.
\end{tabular}

\begin{proposition}\label{correspondencia}
Assume $F$ is of characteristic 0. 
	Let $1\leq t\leq k$ and $n=l_{1}+l_{2}+\cdots + l_{t}+ m$, we define $\psi\colon  V_{n}\to V_{m ,l_{1},\ldots, l_{t}} $ as the linear isomorphism induced by the map 
\[
	x_{i}\mapsto   
	\begin{cases} x_{i}^{1},\text{ if }  1\leq i\leq l_{1} \\  
	x_{i-l_{1}}^{2},\text{ if } l_{1}+1\leq i\leq l_{1}+l_{2} \\ 
	\vdots \\ 
	x_{i-(l_{1}+\cdots + l_{t-1})}^{t},\text{ if } l_{1}+\cdots + l_{t-1}+1\leq i\leq l_{1}+\cdots + l_{t-1}+l_{t} \\ 
	z_{i-(l_{1}+\cdots +l_{t})} , \text{ otherwise } 
	\end{cases}.
\]
	Then	
	\begin{enumerate}
		\item For every choice of $l_{1}$, \dots, $l_{t}$, $m$, one has 
\[
\psi(V_{n}\cap T(E))=V_{m, l_{1},\ldots, l_{t}}\cap T_{\mathbb{Z}}(E^{\infty}).
\]
		\item If $(1\times l_{1})+ (2\times l_{2})+ \cdots + (t\times l_{t})\leq k$ the following equality holds: 
\[
\psi(V_{n}\cap T(E))=V_{m, l_{1},\ldots, l_{t}}\cap T_{\mathbb{Z}}(E^{k^{\ast}}).
\]
	\end{enumerate}
\end{proposition}

\begin{proof}
	The proofs of both statements are quite similar, and that is why we shall prove the second of them. The inclusion $\psi(V_{n}\cap T(E))\subseteq V_{m, l_{1},\ldots, l_{t}}\cap T_{\mathbb{Z}}(E^{k^{\ast}})$ is immediate. Hence we shall prove the opposite inclusion. Let us assume
\[
f(x_{1}^{1},\ldots, x_{l_{1}}^{1},\ldots, x_{1}^{t},\ldots, x_{l_{t}}^{t}, z_{1},\ldots, z_{m} )= \psi(f(x_{1},\ldots, x_{n}))\in T_{\mathbb{Z}}(E^{k^{\ast}}).
\]
	is a multilinear $\mathbb{Z}$-graded polynomial identity for $E^{k^{\ast}}$. We shall prove that $f(x_{1},\ldots, x_{n})\in T(E)$.

	Let $h\colon \{1,\ldots, n\}\to\mathbb{Z}_{2}$ be a fixed map. We use Lemma \ref{market1} to construct a $\mathbb{Z}$-graded evaluation with pairwise disjoint supports,  respecting the degrees given by the map $h$ but vanishing $f$. 
	
	Since $(1\times l_{1})+ (2\times l_{2})+\cdots + (t\times l_{t})\leq k$, we consider disjoint subsets $S_{l_1}$, $S_{l_2}$, \dots, $S_{l_t}$ of $S=\{e_{1}, e_{2},\ldots, e_{k}\}$ such that $S_{l_i}$ has $i\times l_{i}$ elements, for $i=1$, \dots, $t$.
	
Write each set $S_{l_i}$ as the disjoint union of $l_{i}$ subsets of cardinality $i$, that is
	\[
	S_{l_i}=S_{1}^{i}\cup S_{2}^{i}\cup\cdots\cup S_{l_{i}}^{i}.
	\]
		If the integer $s$ is such that $l_{1}+\cdots + l_{i-1}+ 1\leq s\leq l_{1}+\cdots + l_{i}$, we define
	\[
	a_{s}=(\prod_j e_{j}) e_{s}^{\star} \text{ where }e_{j}\in S_{s-(l_{1}+\cdots + l_{i-1})}^{i},
	\]
	and 
	\[
	e_{s}^{\star}=1 \text{ if and only if } i\equiv h(s)\pmod{2\mathbb{Z}},
	\]
	while otherwise the symbol $e_{s}^{\star}$ must be replaced by some generator belonging to $A_{0}$. Since there exist infinitely many generators in $A_{0}$, we choose every consecutive $e_{s}^{\star}$ in such a way that the list $Y=\{e_{s}^{\star}\}$ has no repeated elements, for each $s$. 

	Now we suppose that $n\geq i>l_{1}+\cdots + l_{t}$. In this case we choose monomials $b_{i}$ whose factors do not belong to $S\cup Y$, and such that
	\[
	|supp(b_{i})|\equiv h(i)\pmod{2\mathbb{Z}}\text{ and } b_{1},\ldots, b_{m}\text{ have disjoint supports. }
	\]
It follows that the evaluation
\[
x_j\mapsto a_j, \quad 1\le j\le l_{1}+\cdots + l_{t}; \qquad z_j\mapsto b_j, \quad 1\le j\le m
\]
preserves the map $h$ and also the $\mathbb{Z}$-grading $E^{k^{\ast}}$. Therefore we obtain that
\[
f(a_{1},\ldots, b_{m})=0,
\]
and thus the proof is complete.
\end{proof}

We shall describe a generating set of polynomial identities for the $\mathbb{Z}$-gradings denoted by $E^{k^\ast}$ and $E^{\infty}$. To this end we define the set $D=\{(l_1,\ldots, l_k)\in\mathbb{N}_{0}\times\cdots\times\mathbb{N}_{0}\}$ such that
\[
1\times l_1 + 2\times l_2 +\cdots + k\times l_k\ge k+1.
\]
For each element $(l_1,\ldots, l_k)\in D$, we consider the set $C_{D}$ of   multilinear polynomials
\[x_{1}x_{2}\cdots x_{n},\]
having $l_i$ variables of degree $i$ and $n=l_1+\cdots + l_k$. We observe that each monomial $x_{1}^{t}x_{2}^{t} \cdots x_{[\frac{k}{t}]+1}^{t}$ belongs to $C_{D}$, for $t=1$, \dots, $k$.
It is immediate to verify that all the polynomials above are $\mathbb{Z}$-graded identities for $E^{k^\ast}$. On the other hand every polynomial identity in $C_{D}$ is of homogeneous degree $\geq k+1$, thus it follows from the identities $x$ with $\alpha(x)\notin\{0,\ldots,k\}$.

\begin{theorem}
Assume that $F$ is a field of characteristic 0. 
	Let $T_{\mathbb{Z}}(E^{d})$ be the $T_{\mathbb{Z}}$-ideal of the $\mathbb{Z}$-graded polynomial identities for $E^{d}$, where $d$ stands for either $\infty$ or $k^{\ast}$. Then  
	\begin{enumerate}
		\item $T_{\mathbb{Z}}(E^{\infty})$ is generated by the set of the following polynomials:
		\begin{itemize}
            \item $x$, if $\alpha(x)<0$.
			\item $[x_{1}, x_{2}, x_{3}]$, for every choice of the degrees $\alpha(x_{1})$, $\alpha(x_{2})$, $\alpha(x_{3})$. 
		\end{itemize}
		\item $T_{\mathbb{Z}}(E^{k^{\ast}})$ is generated by the set of the following polynomials:
		\begin{itemize} \label{identidades}
            \item $x$, if $\alpha(x)\notin\{0,\ldots, k\}$. 
            \item $[x_{1}, x_{2}, x_{3}]$, for every choice of the degrees $\alpha(x_{1})$, $\alpha(x_{2})$, $\alpha(x_{3})$.
\end{itemize}
\end{enumerate}
\end{theorem}

\begin{proof}
	 Let $I_{k}$ be the $T_{\mathbb{Z}}$-ideal of $F\langle X|\mathbb{Z}\rangle$ generated by the graded polynomials from Statement~\ref{identidades} of the Theorem. It is easy to show that $I_{k}\subseteq T_{\mathbb{Z}}(E^{k^\ast})$. Given a multilinear $\mathbb{Z}$-graded polynomial identity $f(x_{1}^{1}\ldots,x_{l_k}^{k},\ldots,z_{m})$ of $E^{k^\ast}$ we can suppose that $(1\times l_{1}) + \cdots + (k\times l_{k})\leq k$ and also that the degrees of all variables are $\geq 0$. By Proposition \ref{correspondencia} we have that $f(x_{1},\ldots, x_{n})$ is an ordinary polynomial identity of $E$. Hence $f=\sum_{i=1}^{h}a_{i}[b_{i}, c_{i}, d_{i}]g_{i}$ for some polynomials $a_{i}$, \dots, $g_{i}\in F\langle X\rangle$. As $f (x_1,\ldots, x_n)$ is multilinear we can assume that each of these elements is a monomial in $F\langle X\rangle$,  and $a_{i}b_{i}c_{i}d_{i}g_{i}\in V_{n}$ for every $i=1$, \dots, $h$. Therefore we obtain that
	 $f(x_{1}^{1}\ldots,x_{l_k}^{k},\ldots,z_{m})=\psi(f (x_1,\ldots, x_n))=\sum_{i=1}^{h}\psi(a_{i}[b_{i}, c_{i}, d_{i}]g_{i})=\sum_{i=1}^{h}\overline{a_{i}}[\overline{b_{i}}, \overline{c_{i}}, \overline{d_{i}}]\overline{g_{i}}$. Here $\overline{a_{i}}$, $\overline{b_{i}}$, $\overline{c_{i}}$, $\overline{d_{i}}$, $\overline{g_{i}}$ are monomials in $F\langle X|\mathbb{Z}\rangle$ and $\overline{a_{i}}\overline{b_{i}} \overline{c_{i}}\overline{d_{i}}\overline{g_{i}}\in V_{m ,l_{1},\ldots, l_{t}}$.
	 
	  This proves the second statement of our theorem. The first assertion is proved by using a similar argument.  
\end{proof}

Till the end of this section we consider $F$ an infinite field of characteristic $p>2$. We start with the graded identities of the algebra $E^{k^\ast}$.    
 
The ideas are similar to those in characteristic 0. Nevertheless we cannot restrict ourselves to the multilinear polynomials only but we have to work instead with the multihomogeneous identities.  
 
Most of the results in this section are adaptations of those from \cite{centroneP}. In that paper the author dealt with the $\mathbb{Z}_{2}$-graded case. 

First we state the following theorem that we shall use later on. Its proof is well known, see for example \cite{latyshev, KR}. Formally in these two papers the authors worked over a field of characteristic 0; the argument given is characteristic-free, we refer to \cite{agpk} for further details. 
\begin{theorem}\label{TGP}
Let $E$ be the infinite dimensional Grassmann algebra over an infinite field. Then
\begin{enumerate}
\item The $T$-ideal $I$ of the polynomial identities of $E$ is generated by the triple commutator $[x_1,x_2,x_3]$, that is $I=\langle [x_{1},x_{2},x_{3}]\rangle_{T}$.
\item The polynomials $[x_{1},x_{2}][x_{1},x_{3}]$ and $[x_1,x_2][x_3,x_4]+[x_1,x_3][x_2,x_4]$ belong to $I$.
\end{enumerate}
\end{theorem}

 \begin{proposition}
 If $char(F)=p>0$ then $E^{k^\ast}$ satisfies the graded identity $x^p$ for every degree $\alpha(x)\neq 0$. 
 \end{proposition}  
 \begin{proof}
Follows from the ordinary identities of the Grassmann algebra in characteristic $p>2$, see for example \cite{regevgfa, agpk}.
 \end{proof}
 
\begin{lema}
If $n\in\mathbb{N}$ then the polynomial $t_{2n}=[z_{1},z_{2}][z_{3},z_{4}]\cdots [z_{2n-1},z_{2n}]$ is not a $\mathbb{Z}$-graded identity for $E^{k^\ast}$. Here $\alpha(z_{j})=0$ for every $j$. 
\end{lema}
\begin{proof}
By Theorem~\ref{TGP} we obtain $t_{2n}(e_{k+1},\ldots, e_{k+2n})=2^{n}e_{k+1}\cdots e_{k+2n}\neq 0$. 
\end{proof}

The idea in the following lemma is similar to that  in Lemma 4.4 of \cite{centroneP}. 
\begin{lema}\label{lem 4.1}
Consider $l_{1}, \ldots, l_{k}\in\mathbb{N}_{0}$. Let $T=(x_{1}^{1})^{r_{1}^1}\cdots (x_{l_1}^{1})^{r_{l_1}^1}\cdots (x_{1}^{k})^{r_{1}^k}\cdots (x_{l_k}^{k})^{r_{l_k}^k}$ be a graded monomial satisfying the condition $\sum_{i=1}^{k}i(r_{1}^{i}+\cdots + r_{l_i}^{i})\leq k$. Let $r=\max\{r_{1}^{1},\ldots, r_{l_k}^{k}\}$ and suppose that $p>r$. Then $T$ is not a $\mathbb{Z}$-graded identity of $E^{k^\ast}$. 
\end{lema}

\begin{proof}
For every $i=1$, \dots, $k$ and $t=1$, \dots, $l_i$ we consider the subsets $ S_{i}^{r_{t}^i}\subset\{e_{1},\ldots, e_{k}\}$ such that $|S_{i}^{r_{t}^i}|=ir_{t}^{i}$. Since $\sum_{i=1}^{k}i(r_{1}^{i}+\cdots + r_{l_i}^{i})\leq k$ we can construct such sets so that $S_{i}^{r_{t}^i}\cap S_{j}^{r_{s}^j}=\emptyset$ for every $i$, $j=1$, \dots, $k$, $t=1$, \dots $l_i$, and $s=1$, \dots, $l_j$. Now let $M_{i}^{r_{t}^{i}}$ be a set of $r_{t}^{i}$ monomials of length $i$ with pairwise disjoint supports whose factors belong to $S_{i}^{r_{t}^i}$, respectively.

Consider the evaluation $\varphi\colon F\langle X|\mathbb{Z} \rangle\rightarrow E$ given by:
$x_{t}^{i}\mapsto \sum_{w\in M_{i}^{r_{t}^{i}}}ww^{\ast}$ 
where 
\[
w^{\ast}=\begin{cases} 1,\text{ if $i$ is an even integer } \\ e_{\ast}, \text{ if $i$ is an odd integer} 
\end{cases},
\]
and, in the second case, we choose different generators $e_{\ast}$ outside of $\{e_{1},\ldots, e_{k}\}$ (there exist infinitely many such generators to choose from).

In this case it follows that 
\[
\varphi((x_{1}^{1})^{r_{1}^1}\cdots (x_{l_1}^{1})^{r_{l_1}^1}\cdots (x_{1}^{k})^{r_{1}^k}\cdots (x_{l_k}^{k})^{r_{l_k}^k})=\prod_{i=1}^{k}\prod_{t=1}^{l_i}{r_{t}^{i}\,!}\, m
\]
where $m$ is a nonzero monomial. Since $p>r$ we have that the last evaluation does not vanish. 
\end{proof}

Define the set $D=\{(r_{1}^{1},\ldots, r_{l_1}^{1},\ldots, r_{1}^{k},\ldots, r_{l_k}^{k})\in\mathbb{N}_{0}\times\cdots\times\mathbb{N}_{0}\}\mid  
\sum\limits_{i=1}^{k}i(r_{1}^{i}+\cdots + r_{l_i}^{i})\geq k+1\}$. 

Let $m= (x_{1}^{1})^{r_{1}^1}\cdots (x_{l_1}^{1})^{r_{l_1}^1}\cdots (x_{1}^{k})^{r_{1}^k}\cdots (x_{l_k}^{k})^{r_{l_k}^k}$ where $(r_{1}^{1},\ldots, r_{l_1}^{1},\ldots, r_{1}^{k},\ldots, r_{l_k}^{k})\in D$. Evaluating $m$ on the $e_i\in E^{k^*}$ we have to repeat at least one among $e_1$, \dots, $e_k$ thus vanishing $m$. (It is also easy to see that $m$ is moreover a consequence of the "trivial" identities of the type $x$ with $\alpha(x)\notin\{0,\ldots, k\}$.) 

\begin{theorem}\label{pricipal theo}
Let $p>2$ be a prime and let $k\in\mathbb{N}_{0}$. Over an infinite field $F$ of characteristic $p>k$, all $\mathbb{Z}$-graded polynomial identities of $E^{k^\ast}$ are consequences of the graded identities:
\begin{itemize}
\item 
$x$, if $\alpha(x)\notin\{0,1,\ldots, k\}$;
\item 
$[x_{1},x_{2},x_{3}]$, for every choice of the degrees $\alpha(x_1)$, $\alpha(x_{2})$, $\alpha(x_{3})$.
\end{itemize}
On the other hand, if $p\leq k$ all $\mathbb{Z}$-graded polynomial identities of $E^{k^\ast}$ are consequences of the above two types of graded identities and the identity:
\begin{itemize}
\item $(x^{t})^p$, if $pt\leq k$, for $t\in\{1,\ldots, k\}$.
\end{itemize}
\end{theorem}

\begin{proof}
Suppose that $p>k$. Let $f(x_{1},\ldots, x_{n})$ be a $\mathbb{Z}$-graded polynomial identity for $E^{k^\ast}$. We shall assume that $f$ is multihomogeneous, 0-proper, and that the degree of every variable of $f$ belongs to $\{0,1,\ldots,k\}$. Due to the (graded) identity $[x_{1}, x_{2}, x_{3}]$ we may assume that all the commutators appearing in $f$ are of length 2, that is of the form $[x_{a}, x_{b}]$. Hence $f$ can be reduced to a linear combination of the type 
\[
\sum\beta x_{i_1}^{1}\cdots x_{i_l}^{1}\cdots x_{u_1}^{k}\cdots x_{u_p}^{k}[z_{1},z_{2}]\cdots [z_{s-1},z_{s}][x_{j_1}^{1},x_{j_2}^{1}]\cdots [x_{v_1}^{k},x_{v_2}^{k}]
\]
where the indices are ordered. Note that we are supposing that $s$ is an even number. The case where $s$ is odd is treated in a similar manner. 

Let $h_{i}$ be the number of different variables of degree $i$ appearing in $f$ and suppose that for any $t=1$, \dots, $h_{i}$, we have $\deg_{x_{t}^{i}}(f)=r_{t}^{i}$. By Theorem \ref{TGP} we can assume that $f$ is multilinear in the commutators. Hence we write each summand as
\[
 (x_{1}^{1})^{d_{1}^{1}}\cdots (x_{h_1}^{1})^{d_{h_1}^{1}}\cdots (x_{1}^{k})^{d_{1}^{k}}\cdots (x_{h_k}^{k})^{d_{h_k}^{k}}[z_{1},z_{2}]\cdots [z_{s-1},z_{s}][\overline{x_{1}^1}, \ast]\cdots [\ast, \overline{x_{d_k}^k}].
 \]
Here $d_{t}^{i}\in \{r_{t}^{i}, r_{t}^{i}-1\}$, the notation $\overline{x}$ means that the variable $x$ can be absent,  and the indices are ordered. We may assume that $\sum_{i=1}^{k}i(r_{1}^{i}+\cdots + r_{l_i}^{i})\leq k$. 

Suppose on the contrary that the previous polynomials are linearly dependent modulo $T_{\mathbb{Z}}(E^{k^\ast})$. Then there must exist nonzero coefficients such that
\[
\displaystyle\sum\beta (x_{1}^{1})^{d_{1}^{1}}\cdots (x_{h_1}^{1})^{d_{h_1}^{1}}\cdots (x_{1}^{k})^{d_{1}^{k}}\cdots (x_{h_k}^{k})^{d_{h_k}^{k}}[z_{1},z_{2}]\cdots [z_{s-1},z_{s}][\overline{x_{1}^1}, \ast]\cdots [\ast, \overline{x_{d_k}^k}]\in T_{\mathbb{Z}}(E^{k^\ast}).
\]
Consider the substitution $\varphi\colon F\langle X|\mathbb{Z}\rangle\rightarrow E$ given by $z_{i}\mapsto e_{k+i}$ 
while for the remaining variables we repeat the construction of Lemma \ref{lem 4.1}.  
In the substitution of Lemma \ref{lem 4.1} we replace the variables for central elements. Consequently, the summand
\[
f_{1}=(x_{1}^{1})^{r_{1}^{1}}\cdots (x_{h_1}^{1})^{r_{h_1}^{1}}\cdots (x_{1}^{k})^{r_{1}^{k}}\cdots (x_{h_k}^{k})^{r_{h_k}^{k}}[z_{1},z_{2}]\cdots [z_{s-1},z_{s}]
\]
will be the unique non-vanishing summand of $f$ under the substitution $\varphi$. Hence the coefficient of $f_{1}$ is zero. Using the same argument that was used in \cite{centroneP} we have a  substitution that vanishes all but one given summand of the above linear combination. Thus chosen a non-zero coefficient in it there is a substitution which will send the corresponding monomial to a nonzero element and all remaining monomials vanish. Hence all coefficients are 0, and we are done. 
\end{proof}

In order to describe the $\mathbb{Z}$-graded polynomial identities of $E^{\infty}$ in positive characteristic we note that Lemma \ref{lem 4.1} also holds for $E^{\infty}$. Therefore we have the following lemma.

\begin{lema}\label{lem 4.2}
Consider the graded monomial $T=(x_{1}^{1})^{r_{1}^1}\cdots (x_{h_1}^{1})^{r_{h_1}^1}\cdots (x_{1}^{k})^{r_{1}^k}\cdots (x_{h_k}^{k})^{r_{h_k}^k}$ and denote by $r=\max\{r_{1}^{1},\ldots, r_{h_k}^{k}\}$. If $p>r$  then $T$ is not a $\mathbb{Z}$-graded identity of $E^{\infty}$. 
\end{lema} 
\begin{proof}
We first observe that each homogeneous component of non-negative degree in $E^\infty$ has infinitely many nonzero products of generators $e_i$ of even, respectively of odd, length, and with disjoint supports. Using this it is immediate to exhibit a substitution which does not vanish $T$. 
\end{proof}

By using the above Lemma we describe the $T_{\mathbb{Z}}$-ideal of $E^{\infty}$.

\begin{theorem}
Over an infinite field $F$ of characteristic $p>2$, the $\mathbb{Z}$-graded polynomial identities of $E^{\infty}$ are consequences of the graded identities:
\begin{itemize}
\item $[x_{1},x_{2},x_{3}]$, for every choice of the degrees $\alpha(x_{1})$, $\alpha(x_{2})$, $\alpha(x_{3})$;
\item $x$, if $\alpha(x)<0$;
\item $x^{p}$, if $\alpha(x)\geq 1$. 
\end{itemize}
\end{theorem}
\begin{proof}
One uses the same argument as of Theorem \ref{pricipal theo}. 
\end{proof}

\section{Graded identities of $E^{k}$ in characteristic zero}
\label{sectEk}

In this section we assume the base field of characteristic 0. Recall that the grading $E^k$ is defined by $\|e_i\|=0$ for $1\le i\le k$, and $\|e_i\|=1$ if $i\ge k+1$. In this section we shall stick to the notation $\|\cdot \|$ for the homogeneous degree of an element in a graded algebra. As mentioned earlier we can consider multilinear 0-proper polynomials only. We follow ideas from \cite[Sections 5--8]{disil}. Some of the proofs that are completely analogous to those from \cite{disil} will be omitted. 

We denote by $\Gamma_{l,m}$ the vector subspace of $V_{m,l}$ consisting of the 0-proper multilinear polynomials in the variables $y_1$, \dots, $y_l$ of degree 0 and in the variables $z_1$, \dots, $z_m$ of fixed but arbitrary positive degrees $d_1$, \dots, $d_m$. The notation $\Gamma_{m,l}$ instead of $\Gamma_{l,m}$ could have been more consistent with $V_{m,l}$ but we prefer inverting the indices $l$ and $m$ in order to keep the similarity with that of \cite{disil}. The following proposition is taken from Lemmas 12 and 13 of \cite{disil}. As above we denote by $I$ the ideal of graded identities generated by the triple commutators $[u_1,u_2,u_3]$ for every choice of the degrees of the $u_i$. 

\begin{proposition}
	\label{gamma_lm}
	The vector space $\Gamma_{l,m}$ is spanned, modulo $I$, by the elements:
	
	a) $z_{i_1}\cdots z_{i_m} [y_1,y_2]\cdots [y_{l-1},y_l]$, if $l$ is even. It follows that for every $f\in \Gamma_{l,m}$ one has 
	\[
	f(y_1,\ldots,y_l,z_1,\ldots,z_m) = g(z_1,\ldots,z_m)[y_1,y_2]\cdots [y_{l-1},y_l] \pmod{I}
	\]
	for some polynomial $g\in \Gamma_{0,m}$.
	
	b) $z_{i_1}\cdots z_{i_{m-1}} [z_{i_m},y_1][y_2,y_3]\cdots [y_{l-1},y_l]$, if $l$ is odd and $m\ge 1$. This means that for each $f\in \Gamma_{l,m}$ one has 
	\[
	f(y_1,\ldots,y_l,z_1,\ldots,z_m) = g(z_1,\ldots,z_m, y_1)[y_2,y_3]\cdots [y_{l-1},y_l] \pmod{I}
	\]
	for some polynomial $g\in \Gamma_{1,m}$.
\end{proposition}

The following corollary then is immediate (see Lemmas 14 and 15 of \cite{disil}). 
\begin{corollary}
	\label{even_odd}
	a) Let $l$ be even and $f\in \Gamma_{l,m}$, then for $l\ge k+1$ we have $f\in T_{\mathbb{Z}}(E^k)$. Moreover for $l\le k$ we have $f\in T_{\mathbb{Z}}(E^k)$ if and only if $g\in T_{\mathbb{Z}}(E^{k-l})$.
	
	b) Similarly if $l$ is odd, $m\ge 1$, and $f\in\Gamma_{l,m}$ we get that for $l\ge k+1$ it holds $f\in T_{\mathbb{Z}}(E^k)$. If $l\le k$ then $f\in T_{\mathbb{Z}}(E^k)$ if and only if $g\in T_{\mathbb{Z}}(E^{k-l+1})$.
\end{corollary}

Hence in order to describe the graded identities for $E^k$ it will be sufficient to describe those of $E^{h}$ in $\Gamma_{0,m}$ and in $\Gamma_{1,m}$, for every $m\ge 1$ and $h\ge 0$, and for every choice of the positive integers $d_1$, \dots, $d_m$. Suppose that $m=p+q$ where $p$ and $q$ are non-negative integers, and that $d_1$, \dots, $d_p$ are odd while $d_{p+1}$, \dots, $d_m$ (observe that $p+q=m$) are even. As one would expect, it turns out that the variables $z_{p+1}$, \dots, $z_m$ behave much like the ones of homogeneous degree 0. 

\subsection{Graded identities in $\Gamma_{0,m}$}

The vector space $\Gamma_{0,m}$ is spanned, modulo $I$, by the elements of the type $z_{i_1}\cdots z_{i_r} [z_{j_1}, z_{j_2}]\cdots [z_{j_{t-1}}, z_{j_t}]$ where $t+r=m$ and $t$ is even. Moreover it suffices to consider only the above elements with  $i_1<\cdots<i_r$ and $j_1<\cdots<j_t$. Denote $T=\{j_1,\ldots,j_t\}$, and $f_T = z_{i_1}\cdots z_{i_r} [z_{j_1}, z_{j_2}]\cdots [z_{j_{t-1}}, z_{j_t}]$. 

As in \cite[Definition 16]{disil} one defines the element $g_1(z_1)=z_1$, and for $m\ge 2$, $g_m(z_1,\ldots,z_m) = \sum (-2)^{-|T|/2} f_T$ where the sum runs over all subsets $T\subseteq \{1,\ldots, m\}$ such that $|T|$ is even. 

\begin{proposition}\cite[Proposition 18]{disil}
	\label{g_heven}
	Let $\|z_i\|$ be odd positive integers, $i=1$, \dots, $h+2$. Then the polynomial $g_{h+2}(z_1, \ldots, z_{h+2})$ is a graded identity for $E^h$.
\end{proposition}
We draw the readers' attention that we require \textsl{all} $z_i$ to be of odd homogeneous degrees. Then the proof of Proposition 18 from \cite{disil} applies word by word to our case. 

A further reduction of the type of identities satisfied by $E^h$ is due. 

\begin{lema}
	\label{further_red}
	Let $q\ge 1$ and suppose that $f\in \Gamma_{0,m}\cap T_{\mathbb{Z}}(E^h)$ is a graded identity for $E^h$. Then $f\equiv f_1z_m + f_2\pmod{I}$ where $f_1\in \Gamma_{0, m-1}\cap T_{\mathbb{Z}}(E^h)$ and $z_m$ appears inside commutators only in $f_2$. Moreover $f_2\in \Gamma_{0,m}\cap T_{\mathbb{Z}}(E^h)$.
\end{lema} 

\begin{proof}
One uses several times the equality $uz_mz_iv = u[z_m,z_i]v + uz_iz_mv$ where $u$ and $v$ are monomials, and the fact that $[z_m,z_i]$ are central elements. Thus one obtains $f\equiv f_1z_m + f_2\pmod{I}$. 
	We have that $\|z_m\|>0$ is an even integer. There are infinitely many elements of the canonical basis of $E$ in the homogeneous component of degree $\|z_m\|$ whose length is even, thus these elements are central. Substitute $z_m$ for such an element of the canonical basis of $E$, this vanishes $f_2$. If one substitutes now $z_1$, \dots, $z_{m-1}$ for elements of the canonical basis of $E$ and taking care of the supports to be disjoint one gets $f_1\in \Gamma_{0, m-1}\cap T_{\mathbb{Z}}(E^h)$. Then $f_1z_m\in \Gamma_{0,m}\cap T_{\mathbb{Z}}(E^h)$ and this implies the last statement of the lemma.
\end{proof}

\begin{corollary}
	\label{cor_further_red}
	Let $q\ge 1$ and suppose $f\in \Gamma_{0,m}\cap T_{\mathbb{Z}}(E^h)$ is a graded identity for $E^h$. 
	
	If $q$ is even then $f\equiv f'(z_1, \ldots, z_p) [z_{p+1},z_{p+2}] \cdots [z_{m-1},z_m]\pmod{I}$. In case $h\ge q$ then $f'(z_1,\ldots, z_p)$ is a graded identity for $E^{h-q}$ which depends on variables of odd homogeneous degree. 
	
	If $q$ is odd then $f\equiv f'(z_1,\ldots, z_p,z_{p+1}) [z_{p+2},z_{p+3}] \cdots [z_{m-1},z_m]\pmod{I}$. In case $h\ge q-1$ then $f'(z_1,\ldots, z_p)$ is a graded identity for $E^{h-q+1}$ which depends on only one variable of even homogeneous degree. 
\end{corollary}

Here we observe that one cannot apply directly 0-proper identities in this situation (the element 1 does not lie in the corresponding homogeneous components). That is why we have to do the straightforward computation. 

Therefore in order to describe the graded identities for $E^h$ in $\Gamma_{0,m}$ it suffices to know the graded identities for $E^s$ for each $s$, in $\Gamma_{0,m}$ where $q=0$ or 1. Recall that among $d_i=\|z_i\|$ there are $p$ odd and $q$ even positive integers. 

Suppose first $q=0$. As in Definition 19 of \cite{disil} we denote by $J^h$ the ideal of graded identities defined by $I$ and by the polynomial $g_{h+2}(z_1, \ldots, z_{h+2})$ where all $\|z_i\|=d_i$ are odd. Then one has the analogue of Proposition 25 (a) from \cite{disil}, with exactly the same proof. 

\begin{proposition}
	\label{00q}
	Let all $z_i$, $1\le i\le m$ be of odd homogeneous degree. Then for every $m\ge 1$ one has $\Gamma_{0,m}\cap T_{\mathbb{Z}}(E^h) = \Gamma_{0,m}\cap J^h$.
\end{proposition} 

Let $q=1$, that is we have in $f(z_1, \ldots,z_m)$ only one variable, $z_m$ of even homogeneous degree $d_m>0$. This case is very similar to that of only one variable $y$ dealt with in Section 7 of \cite{disil}. 

We follow the notation from Section 7 in \cite{disil}. The vector space $\Gamma_{0,m}$ is spanned in this case, modulo $I$, by the polynomials 
\[
f_T'(z_1,\ldots,z_m) = z_{i_1}\cdots z_{i_r} [z_{j_1},z_{j_2}] \cdots [z_{j_t},z_m].
\]
Here $i_1<\cdots<i_r$, $j_1<\cdots<j_t$, $r+t=m$ and $T=\{j_1,\ldots, j_t\}$, $|T|=t$ is an odd integer. 

\begin{proposition}\cite[Section 7]{disil}
	\label{g_hodd}
	Let $\|z_i\|$ be odd integers for $1\le i\le h+1$, and let $\|z_{h+2}\|$ be even. Then the polynomial $[g_{h+1}(z_1, \ldots, z_{h+1}), z_{h+2}]$ is a graded identity for $E^h$. 
	
	If $\|z_i\|$ are odd for $1\le i\le h+2$ and $\|z_{h+3}\|$ is even then $g_{h+1}(z_1, \ldots, z_{h+1})[z_{h+2},z_{h+3}]$ is a graded identity for $E^h$.
\end{proposition}

We denote by $J_1^h$ the ideal of graded identities generated by $I$ and by the two polynomials from Proposition~\ref{g_hodd}. Then the contents of Section 7 in \cite{disil} transfers literally to our case (substituting the variable $y$ for the corresponding unique variable $z$ of even positive homogeneous degree). In particular Proposition 28 and Lemma 31 from \cite{disil} hold in our case with the due changes. 
We will need explicitly the following proposition.

\begin{proposition}\cite[Proposition 32 (a)]{disil}
	\label{gh_j1}
	Suppose $h\ge 0$ and $m\ge 1$. Let $\|z_i\| = d_i$ be odd positive integers for $1\le i\le m-1$ and let $d_m>0$ be even. Then $\Gamma_{0,m}\cap T_{\mathbb{Z}}(E^h) = \Gamma_{0,m}\cap J_1^h$.
\end{proposition}

\subsection{Graded identities in $\Gamma_{1,m}$}
This case is quite similar to the one already considered in the previous subsection so we outline only the differences. As in the case of $\Gamma_{0,m}$ one obtains that, modulo $I$, the vector space $\Gamma_{1,m}$ is spanned by the elements of the type
\[
f'_T = z_{i_1}\cdots z_{i_r} [z_{j_1},z_{j_2}] \cdots [z_{j_{t-2}}, z_{j_{t-1}}][z_{j_t},y].
\]
Here as above $i_1<\cdots<i_r$, $j_1<\cdots<j_t$, $T=\{j_1,\ldots,j_t\}$ and $t$ is an odd integer. Denote as above $d_i=\|z_i\|$. The $d_i$ are positive integers. 
As in the previous subsection we have to consider the two possibilities: when all $d_i$ are odd, and when there is exactly one even among them.

Suppose all $d_i$ are odd. Then the whole Section 7 of \cite{disil} applies without any change. We recall very briefly the results from \cite{disil} we need here. The polynomials 
\[
[g_{h+1}(z_1, \ldots, z_{h+1}), y], \qquad g_{h+1}(z_1, \ldots, z_{h+1})[z_{h+2},y]
\]
are graded identities for $E^h$. If we put $J_2^h$ the ideal of graded identities generated by $I$ together with these two polynomials one obtains, repeating word by word the argument from \cite{disil}, the following proposition.

\begin{proposition}\cite[Proposition 32 (a)]{disil}
	\label{gh_j2}
	For $m\ge 1$ and each $h\ge 0$, one has $\Gamma_{1,m}\cap T_{\mathbb{Z}}(E^h) = \Gamma_{1,m}\cap J_2^h$ where all variables $z_i$, $1\le i\le m$ are of odd positive homogeneous degree.
\end{proposition}

Now let the $d_i$, $1\le i\le m-1$ be odd and $d_m$ be even. As it was done earlier we can consider, modulo $I$, that $z_m$ appears inside commutators only. Also modulo $I$ one has that $[x_1,x_2][x_3,x_4] = -[x_1,x_3][x_2,x_4]$ (this is an ordinary identity for the Grassmann algebra). Therefore we can suppose that $z_m$ and $y$ are in the same commutator. In this way we have that every polynomial $f$ in $\Gamma_{1,m}$ can be written as 
\[
f(z_1,\ldots,z_m,y) \equiv f'(z_1,\ldots,z_{m-1}) [z_m, y] \pmod{I}.
\]
It follows that $f$ is a graded identity for $E^h$ if and only if $f'$ is one for $E^{h-1}$ for every $h\ge 1$. (Since $d_m$ is even we have to use one of the $h$ basic elements assumed of degree 0 in order to change the parity of $z_m$ under an evaluation, otherwise the commutator vanishes modulo $I$.) All variables $z_1$, \dots, $z_{m-1}$ are of odd positive degree. 

Let $d_i$, $1\le i\le h+1$ be odd and $d_{h+2}$ be even. The graded identities 
\[
[g_{h+1}(z_1,\ldots, z_{h+1}), y], \qquad g_{h+1}(z_1,\ldots, z_{h+1})[z_{h+2},y]
\]
are satisfied by $E^h$. We denote by $J_3^h$ the ideal of graded identities generated by $I$ and by the above polynomials. As in \cite[Proposition 32 (a)]{disil} we obtain
\begin{proposition}
	\label{gh_j3}
	Suppose $m\ge 2$ and $d_i=\|z_i\|$ are odd whenever $1\le i\le m-1$ while $d_m=\|z_m\|$ is even. Then $\Gamma_{1,m}\cap T_{\mathbb{Z}}(E^h) = \Gamma_{1,m}\cap J^h_3$.
\end{proposition}
In order to describe the graded identities for the $\mathbb{Z}$-graded algebra $E^k$ we define the ideal $P^k$ of graded identities as the one generated by the following polynomials. 

\begin{enumerate}
	\item
	$x$ whenever $\|x\|<0$.
	\item
	$[x_1, x_2, x_3]$, for every choice of the degrees $\|x_i\|$.
	\item
	$[y_1,y_2]\cdots [y_{k-1},y_k][y_{k+1}, x]$ where $\|x\|$ is arbitrary, if $k$ is even.
	\item
	$[y_1,y_2]\cdots [y_{k-2},y_{k-1}][y_k, y_{k+1}]$, if $k$ is odd.
	\item
	$g_{k-l+2}(z_1, \ldots, z_{k-l+2}) [u_1,u_2] \cdots [u_{l-1}, u_l]$ where $\|z_i\|$ are positive odd integers and $\|u_i\|$ are even integers (some of the $u_i$ may be variables $y$ while the remaining are $z$ of even homogeneous degree). Here $l\le k$ is even.
	\item
	$[g_{k-l+2}(z_1, \ldots, z_{k-l+2}), u_1] [u_2, u_3]\cdots [u_{l-1}, u_l]$ where $\|z_i\|$ are positive odd integers and $\|u_i\|$ are even integers (some of them $y$ and the remaining $z$ with $\|z\|$ even). Here $l\le k$ is odd.
	\item
	$g_{k-l+2}(z_1, \ldots, z_{k-l+2}) [z_{k-l+3}, u_1] \cdots [u_{l-1}, u_l]$ where as above $\|z_i\|$ are positive odd integers and $\|u_i\|$ are even integers (some of them $y$ and the remaining $z$ with $\|z\|$ even). Here $l\le k$ is odd.
\end{enumerate}

Depending on $k$ being even or odd we include in $P^k$ only the corresponding identities in items (3) or (4) above.

\begin{lema}
	\label{easy_incl}
	The polynomials above are graded identities for the $\mathbb{Z}$-graded algebra $E^k$, that is $P^k\subseteq T_{\mathbb{Z}}(E^k)$. 
\end{lema}

\begin{proof}
	The first two types of polynomials are obviously graded identities for $E^k$. If $k$ is even then the polynomial in item 3 is an identity since there are just $k$ linearly independent elements of homogeneous degree 0, namely $e_1$, \dots, $e_k$ (and one cannot put 1 in a commutator without vanishing it). The same argument applies to item 4.   Items 5, 6, 7 are dealt with in a similar manner so we consider item 5 only. In it, if some of the $u_i$ is a variable $y$ of degree 0 then we are left with $k-1$ elements among $e_1$, \dots, $e_k$. But when substituting those among the $u_i$ which are of positive even degree with elements of $E^h$ we have to use for each of them at least one among the $e_1$, \dots, $e_k$. Otherwise the corresponding $z$ will be substituted with a central element of $E$ thus annihilating the commutator.
\end{proof}

Here is the main theorem of the section. Its proof is now a straightforward adaptation of the proof of Theorem 38 in \cite{disil}, and we omit it.

\begin{theorem}
	\label{mainthm_eh}
	The ideal $T_{\mathbb{Z}}(E^k)$ of the $\mathbb{Z}$-graded identities for the algebra $E^k$ coincides with the ideal $P^k$.
\end{theorem}

\section{Further examples of $\mathbb{Z}$-gradings on $E$}

\subsection{$\mathbb{Z}$-gradings of the form $(r)$, $(\infty)$}

Let $r\in\mathbb{N}$ and put $\|e_{n}\|=r$  for every $n\in\mathbb{N}$, we obtain the $\mathbb{Z}$-grading $E_{(r)}^{(\infty)}=\oplus_{n\in\mathbb{Z}} A_{n}$ such that
\[ 
A_{n}=\begin{cases} 0,\text{ if }  n<r \text{ and }n\neq 0 \\ F, \text{ if } n=0 \\ span_{F}\{w\in B_{E}  \mid  |supp(w)|=m \}, \text{ if } n=rm \\ 0,\text{ if } n>r \text{ and } r\nmid n
\end{cases}.
\]
Observe that this grading has support $C=\{0, r, 2r,\ldots, nr,\ldots\}$.  
We shall describe the $\mathbb{Z}$-graded polynomial identities for this grading. An easy and direct computation shows that the following polynomials are $\mathbb{Z}$-graded polynomial identities of $T_{\mathbb{Z}}(E_{(r)}^{(\infty)})$:
\begin{itemize}
\item $x$, if $\alpha(x)\notin C$;
\item $[x_{1}^{2rn},x_{1}^{rm}]$, for every $n$, $m\in\mathbb{N}_{0}$;
\item $x_{1}^{rn}x_{1}^{rm}+x_{1}^{rm}x_{1}^{rn}$, if both $n$ and $m$ are odd integers. 
\end{itemize}

\begin{example}
Let $E$ be the $\mathbb{Z}$-graded algebra just constructed. The ideal of graded identities $T_{\mathbb{Z}}(E_{(r)}^{(\infty)})$ is generated by the identities above.  The argument repeats word by word that of the natural $\mathbb{Z}$-graded case. 
\end{example}

Let $H$ be the subgroup of $\mathbb{Z}$ generated by the
	given element $r$. Clearly $E$ becomes an algebra graded by $H$, and this new $H$-grading is
	equivalent to the $\mathbb{Z}$-grading $E^{can}$. In this context we say that $E_{(r)}^{(\infty)}$ and $E^{can}$ are equivalent gradings. In fact the graded identities are essentially the ``same'', re-scaling by the multiple $r$. It is also easy to see that in characteristic $p>2$ one has to add to the list of identities of the above theorem the $p$-th powers of the elements $x_1^{nr}$ with $n$ even. Note once again that if $n$ is odd the square of the corresponding variable vanishes thus its $p$th power is 0 as well.

\subsection{$\mathbb{Z}$-gradings of the forms $(p,q)$, $(1,\infty)$, and $(p,q)$, $(k,\infty)$}
Here we give, as examples, the graded identities for the gradings from the title. 

Let $p$, $q$ be primes such that $p<q$. Consider the decomposition 
\[
L=L_{p}^{1}\oplus L_{q}^{\infty}.
\]
We suppose that $\dim L_{p}=1$ and $\dim L_{q}=\infty$. Up to a change of the basis we can assume that $L_{p}$ is the span of $\{e_{1}\}$ and $L_{q}$ is the span of $\{e_{2},e_{3},\ldots\}$. In the above list notation this decomposition provides us with the $\mathbb{Z}$-grading $E_{(p,q)}^{(1,\infty)}=\oplus_{n\in\mathbb{Z}} A_{n}$.

Now we describe its components and its $\mathbb{Z}$-graded polynomial identities. First we define the set
\[
C=\{px+qy\mid x=0,1\text{ and } y\in\mathbb{N}_{0}\}.
\]
We define the following sets:
\begin{align*}
C_{1}=\{qy\mid y\text{ is even}\}, & \quad
C_{2}=\{qy\mid y\text{ is odd } \},\\
C_{3}=\{p+qy\mid y\text{ is even }\}, & \quad
C_{4}=\{p+qy\mid y\text{ is odd }\}.
\end{align*}

Obviously we have that
$$C=C_{1}\cup C_{2}\cup C_{3}\cup C_{4}.$$ 

We denote by $E_{can}=E_{(0)}\oplus E_{(1)}$ the natural $\mathbb{Z}_{2}$-grading on $E$. If $y\in\mathbb{N}$ we  adopt the notation
\[
E^y(e_n)=\text{ the subspace of $E$ spanned by all products of $e_i$'s of length } y \text{ without the factor } e_{n};
\]
\[
E^y[e_n]=\text{ the subspace of $E$ spanned by all products of $e_i$'s of length } y+1 \text{ with the factor } e_{n}.
\]
Using the previous notation we have that the description of the $\mathbb{Z}$-grading $E_{(p,q)}^{(1,\infty)}$ is given by the following rules. Here $A_n$ stands for the homogeneous component of degree $n$.
\[
A_{n}=\begin{cases} 0,\text{ if }  n\notin C \\ F, \text{ if } n=0 \\ E_{(0)}\cap E^{y}(e_{1}), \text{ if } n\in C_{1} \text{ and }n=qy \\ E_{(1)}\cap E^{y}(e_{1}), \text{ if } n\in C_{2} \text{ and }n=qy \\ E_{(1)}\cap E^{y}[e_{1}], \text{ if } n\in C_{3} \text{ and }n=p+qy \\ E_{(0)}\cap E^{y}[e_{1}], \text{ if } n\in C_{4} \text{ and }n=p+qy
\end{cases}, y\neq 0.
\]
\begin{example}\label{E_{p,q}}
It is easy to check that the $\mathbb{Z}$-graded algebra $E_{(p,q)}^{(1,\infty)}$ satisfies the following $\mathbb{Z}$-graded polynomial identities: 
\begin{itemize}
\item $x$, if $\alpha(x)\notin C$;
\item $[x_{1},x_{2}]$, if $\alpha(x_{1})\in C_{1}\cup C_{4}$;
\item $x_{1}x_{2}+x_{2}x_{1}$, if $\alpha(x_{1})$, $\alpha(x_{2})\in C_{2}\cup C_{3};$
\item $x_{1}x_{2}$, if $\alpha(x_{1})$, $\alpha(x_{2})\in C_{3}\cup C_{4}$.
\end{itemize}
In the following example we shall consider a more general situation which will imply that the  $T_{\mathbb{Z}}$-ideal $T_{\mathbb{Z}}(E_{(p,q)}^{(1,\infty)})$ is generated by the graded identities above. 
\end{example}

\begin{example}
Let $p$, $q$ be primes such that $p<q$. We consider once again the decomposition 
\[
L=L_{p}^{k}\oplus L_{q}^{\infty}.
\]
We suppose that $\dim L_{p}^{k}=k<p$ is an integer and $\dim L_{q}^{\infty}=\infty$. Up to a change of basis we can assume that $L_{p}^{k}$ is the span of $\{e_{1},\ldots,e_{k}\}$ and $L_{q}^{\infty}$ is the span of $\{e_{k+1},e_{k+2},\ldots\}$.

In this way we have the $\mathbb{Z}$-grading $E_{(p,q)}^{(k,\infty)}=\oplus_{n\in\mathbb{Z}} A_{n}$. Here $A_n$ is the homogeneous component of degree $n$. Once again we define 
\[
C=\{px+qy\mid x=0,1,\ldots, k\text{ and } y\in\mathbb{N}_{0}\}.
\]
We form the sets 
\begin{align*}
C_{1}=\{qy\mid y\text{ is even}\}, & \quad
C_{2}=\{qy\mid y\text{ is odd}\},\\
D_{i}=\{pi+qy\mid y\equiv (i+1)\pmod{\mathbb{Z}_{2}}\}, & \quad 
\hat{D_{i}}=\{pi+qy\mid y\equiv i\pmod{\mathbb{Z}_{2}}\},\\
\end{align*}
for $i=1$, \dots, $k$.
Obviously we have 
\[
C=C_{1}\cup C_{2}\cup D_{1}\cup\hat{D_{1}}\cdots D_{k}\cup  \hat{D_{k}}.
\]
Here we use the notation
\begin{align*}
{E_{t}}^y[e_{1},\ldots, e_{k}]=&\text{ the set of $E$ formed by all products of $e_i$'s of length } y+t\text{ with $t$ factors in } \{e_{1},\ldots, e_{k}\}.\\
E^y(e_{1},\ldots, e_{k})=&\text{ the set of $E$ formed by all products of $e_i$'s of length } y\text{ without factors in } \{e_{1},\ldots, e_{k}\}.
\end{align*}
Therefore we have the following description of the homogeneous components:
\[
A_{n}=\begin{cases} 0,\text{ if }  n\notin C \\ F, \text{ if } n=0 \\
E_{(0)}\cap E^{y}(e_{1},\ldots,e_{k}), \text{ if } n\in C_{1} \text{ and }n=qy \\ 
E_{(1)}\cap E^{y}(e_{1},\ldots,e_{k}), \text{ if } n\in C_{2} \text{ and }n=qy \\ 
E_{(1)}\cap E_{1}^{y}[e_{1},\ldots,e_{k}], \text{ if } n\in D_{1} \text{ and }n=p+qy \\ 
E_{(0)} \cap E_{1}^{y}[e_{1}, \ldots,e_{k}], \text{ if } n\in \hat{D_{1}} \text{ and }n=p+qy \\ 
\vdots \\ 
E_{(1)}\cap E_{k}^{y}[e_{1},\ldots,e_{k}], \text{ if } n\in D_{k} \text{ and }n=kp+qy \\ 
E_{(0)}\cap E_{k}^{y}[e_{1},\ldots,e_{k}], \text{ if } n\in \hat{D_{k}} \text{ and }n=kp+qy
\end{cases}.
\]
We assume above that $y\neq 0$. The description of the graded identities of $E_{(p,q)}^{(k,\infty)}$ is given below. First it is immediate to check that the following polynomials are graded identities for the $\mathbb{Z}$-graded algebra $E_{(p,q)}^{(k,\infty)}$. Here we assume $p$, $q$ are primes and $k<p<q$. 

\begin{itemize}
\item $x$, if $\alpha(x)\notin C$;
\item $[x_{1},x_{2}]$, if $\alpha(x_{1})\in C_{1}\cup\hat{D_{1}}\cup\cdots\cup\hat{D_{k}};$
\item $x_{1}x_{2}+x_{2}x_{1}$, if $\alpha(x_{1})$, $\alpha(x_{2})\in C_{2}\cup D_{1}\cup\cdots\cup D_{k};$
\item $x_{1}\cdots x_{l_1} u_{1}\cdots u_{l_2}\cdots w_{1}\cdots w_{l_k}$, if $\alpha(x)\in D_{1}\cup\hat{D_{1}},\ldots,\alpha(w)\in D_{k}\cup\hat{D_{k}}$ and \linebreak $1\times l_{1}+\cdots + k\times l_{k}>k$. 
\end{itemize}

Now we prove that these graded identities generate the $T_{\mathbb{Z}}$-ideal of the $\mathbb{Z}$-graded algebra $E_{(p,q)}^{(k,\infty)}$. Let $I$ be the ideal of graded identities generated by the above polynomials. As the field is of characteristic 0 we consider multilinear polynomials only. 

Assume that $f(z_{1},\ldots,z_{n})$ is a multilinear $\mathbb{Z}$-graded polynomial and that $f\notin I$. We shall show that $f$ does not vanish on $E_{(p,q)}^{(k,\infty)}$. Clearly we can assume that $\alpha(z_{i})\in C$ for every $i=1$, \dots, $n$. Modulo the identities $[x_{1},x_{2}]$ and $x_{1}x_{2}+x_{2}x_{1}$ there exists $\beta\neq 0$ such that:
\[
f=\beta x_{1}\cdots x_{l_1} u_{1}\cdots u_{l_2}\cdots w_{1}\cdots w_{l_k}\theta_{1}\cdots\theta_{s}
\]
where
\[
\alpha(x)\in D_{1}\cup\hat{D_{1}}, \quad \ldots, \quad \alpha(w)\in D_{k}\cup\hat{D_{k}}, \quad \alpha(\theta)\in C_{1}\cup C_{2}.
\]
Due to the last identities that define $I$ we have that
\[
1\times l_{1}+\cdots + k\times l_{k}\leq k.
\]
Now it is easy to construct a substitution that does not vanish $f$.  
\end{example}

\section{Further discussion}

Here we raise several open questions related to our results.
\begin{enumerate}
\item If the $\mathbb{Z}$-gradings $E_{(n_{1},\ldots, n_{l})}^{(v_{1},\ldots, v_{l})}$ and $E_{(m_{1},\ldots, m_{t})}^{(u_{1},\ldots, u_{t})}$ are isomorphic what can one say about the lists $(n_{1},\ldots, n_{l})$, $(v_{1},\ldots, v_{l})$, and $(m_{1},\ldots, m_{t})$, $(u_{1},\ldots, u_{t})$?

\item What do the $\mathbb{Z}$-graded polynomial identities of $E^{k^\ast}$, $E^{\infty}$ and $E^{k}$ over a finite field look like? 

\item
What are the $\mathbb{Z}$-graded identities of $E^k$ if the base field is infinite of positive characteristic $p>2$?

\item What happens with the $\mathbb{Z}$-gradings on $E$ that are not of finite coverage? For example, we can define a $\mathbb{Z}$-grading on $E$ by the index, that is:
\[
\displaystyle \|e_{i}\|=i,
\]
for every $i\in\mathbb{N}$. 
This $\mathbb{Z}$-grading has an interesting combinatorial interpretation. What are the respective graded polynomial identities? 
\end{enumerate}

\bigskip

\begin{center}
\textbf{Acknowledgment}
\end{center}

\noindent
We express our thanks to the Referee whose comments were taken into account. We are particularly grateful for the Referee's suggestion to work on and include Section 6 of the present paper.

\end{document}